\documentclass[10pt]{amsart}
\usepackage{graphicx}
\usepackage{latexsym}
\usepackage{fancyhdr}
\usepackage{amsmath, amssymb}
 \usepackage[utf8]{inputenc}
\usepackage[all]{xy}
\usepackage{pdflscape}
\usepackage{longtable}
\usepackage{rotating}
\usepackage{subfigure}
\usepackage{tensor}
\usepackage{enumerate}
\usepackage{hyperref}

\usepackage{color}

\theoremstyle{plain}
\newtheorem{thm}{Theorem}[section]

\newtheorem{lem}[thm]{Lemma} 
\newtheorem{prop}[thm]{Proposition}

\theoremstyle{definition}

\theoremstyle{remark}
\newtheorem{rem}[thm]{Remark}
\numberwithin{equation}{section}

\newtheorem*{convention*}{Convention}

\newcommand{\lgw}{\longrightarrow}
\newcommand{\lgm}{\longmapsto}
\newcommand{\si}{\sigma}

\newcommand{\R}{\mathbb{R}}
\newcommand{\K}{\mathbb{K}}

\newcommand{\N}{\mathbb{N}}

\newcommand{\C}{\mathbb{C}}

\newcommand{\Q}{\mathbb{Q}}
\renewcommand{\t}{\tau}

\renewcommand{\lg}{\langle}
\newcommand{\rg}{\rangle}

\renewcommand{\phi}{\varphi}
\renewcommand{\d}{\delta}

\newcommand{\e}{\varepsilon}

\let\mathscr\mathcal

\DeclareMathOperator{\re}{Re}
\DeclareMathOperator{\im}{Im}
\DeclareMathOperator{\mult}{mult }

\newcommand{\ps}{\psi}
\newcommand{\ph}{\phi}
\newcommand{\Ps}{\Psi}
\newcommand{\Ph}{\Phi}

\begin{document}
\title[Higher order approximation of analytic sets]{Higher order approximation of analytic sets by topologically equivalent algebraic sets}

\author{Marcin Bilski}
\address{Department of Mathematics and Computer Science, Jagiellonian University, \L ojasie\-wicza 6, 30-348 Krak\'ow, Poland}
\email{marcin.bilski@im.uj.edu.pl}

\author{Krzysztof Kurdyka}
\address{Universit\'e  Savoie Mont Blanc, CNRS, LAMA, UMR 5127, 73376 Le Bourget-du-Lac, France}
\email{Krzysztof.Kurdyka@univ-savoie.fr}

\author{Adam Parusi\'nski}
\address{Universit\'e Nice Sophia Antipolis, CNRS, LJAD, UMR 7351, 06108 Nice, France}
\email{adam.parusinski@unice.fr}

\author{Guillaume Rond}
\address{Aix-Marseille Universit\'e, CNRS, Centrale Marseille, I2M, UMR 7373, 13453 Marseille, France}
\email{guillaume.rond@univ-amu.fr}

\thanks{The authors were partially supported by ANR project STAAVF (ANR-2011 BS01 009). G. Rond was partially supported by ANR project SUSI (ANR-12-JS01-0002-01). M. Bilski was partially supported by the NCN grant 2014/13/B/ST1/00543.}

%\date{January 11, 2014}


\begin{abstract}
It is known that every germ of
an analytic set is  homeomorphic to the germ of an algebraic set.
In this paper we show that the homeomorphism can be chosen in such a way
that the analytic and algebraic germs are tangent with any prescribed order of tangency.
Moreover, the space of arcs contained in the algebraic germ approximates 
the space of arcs contained in the analytic one, in the sense that they are identical up to a 
prescribed truncation order.  
\end{abstract}

%%%%%%%%%%%%%%%%%%%%%%%%%%%%%%%%%%%%%%%%%%%%%%%%%%%%%%
\maketitle

\section{Introduction}
The problem of approximation of analytic sets (or functions) by algebraic ones  
 is one of the most fundamental problems in singularity theory and an old subject of investigation
(see e.g. \cite{Ar}, \cite{BPR}, \cite{bochnakkucharz1984}, \cite{bochnak1984},  \cite {Ku}, \cite{Lempert}, \cite{mostowski}, \cite{Sa}, \cite{TJC}, \cite{whitney}).  
In general this problem has been considered by two different approaches.  

Firstly one may seek to 
approximate the germs of analytic sets by the germs of algebraic ones so that both objects are homeomorphic.  
 For instance, by  \cite{Ar} and \cite{TJC},  a germ of a coherent analytic set  with  an isolated singularity is analytically equivalent to a germ of an algebraic set (see e.g. Theorem 3 of \cite{bochnak1984} for a precise statement). 
 But in general an analytic set is not even locally diffeomorphically equivalent to an algebraic one (cf. \cite{whitney}). 
Nevertheless by a result of T. Mostowski  \cite{mostowski}  every analytic set is locally topologically equivalent to an algebraic one.

\begin{thm}[\cite {mostowski}, \cite {BPR}] \label{MoThm} 
Let $\K=\R$ or $\C$. Let $(X, 0)\subset(\K^n,0)$ be an analytic germ. Then there is a
homeomorphism $h:(\K^n , 0)\rightarrow (\K^n,0)$ such that $h(X)$ is the germ at $0$ of an algebraic subset of
$\K^n$.
\end{thm}

Secondly one may seek to approximate analytic germs by the algebraic ones with the  higher order tangency,  cf. e.g. \cite{B1}, \cite{B2}, \cite{BMT1}, \cite{FFW4}, \cite{FFW3}.   But the classical methods of such approximation do not provide the objects which are homeomorphic.  
In this paper we show how to construct approximations that satisfy  both requirements that is approximate with homeomorphic objects and with a given order of tangency.  

\begin{thm}\label{Main}
Let $\K=\R$ or $\C$. Let $(X, 0)\subset(\K^n,0)$ be an analytic germ. 
Then  there are $C,c>0$ and an open neighborhood $U$
of $0$ in $\K^n$ such that for every $m\in\mathbf{N}$
there are a subanalytic and arc-analytic homeomorphism 
$\phi_m: U\rightarrow\phi_m(U)\subset\K^n$
and an algebraic subset $V_m$ of $\K^n$ 
with the following properties:\vspace*{2mm}\\
(a)  $\phi_m(X\cap U)=V_m\cap\phi_m (U),$\\
(b) $||\phi_m(a)-a||\leq C^m||a||^m$ for every $a\in U$.
\vspace*{2mm}

Moreover, there is a nowhere dense analytic subset  $Z\subset U$ independent of $m$ 
such that $\varphi_m$ is real analytic on $U\setminus Z$ and such that the jacobian determinant of $\varphi_m$ on $U\setminus Z$ satisfies 
\vspace*{2mm}\\
(c)  $c \le |jacdet (\varphi_m) (x)| \le C$. 

\vspace*{2mm}
\end{thm}

In the above statement, and throughout the paper, "analytic" means "complex analytic" 
if $\K=\C$ and "real analytic" if $\K=\R$, unless we say explicitly "real analytic" or "complex analytic".

The proof of Theorem \ref{Main} is based on Mostowski's approach, its recent refinement  \cite{BPR}, and 
a new result on Zariski equisingularity  given in \cite{parusinski-paunescu}.  
This proof  will be divided into two parts. Firstly, using Popescu Approximation Theorem and a strong version of Varchenko's Theorem given in \cite{parusinski-paunescu},  we show that there is
an arc-analytic homeomorphism $\psi$ such that $\psi(X)$ is a Nash set tangent to $X$ at
$0$ with any prescribed order of tangency (cf. Proposition \ref{Prop1}, Section \ref{NashSect}).
Next, using Artin-Mazur's Theorem of \cite{AM} (cf. also \cite{bochnakkucharz1984}),  we prove that there is a Nash analytic diffeomorphism $\theta$ such that $\theta(\psi(X))$ is an algebraic
set tangent to $\psi(X)$ (cf. Proposition \ref{Prop2}, Section \ref{AlgSect}).

In Section \ref{section_cor}  we give an application of Theorem \ref{Main} to the space of analytic arcs 
on a given germ of analytic set. Namely, we show that for the space of truncated arcs there is 
an algebraic germ, homeomorphic to the original one, with the identical space of truncated arcs.  
We shall use the following notation.  For any analytic germ $(Z,0)\subset(\K^n,0),$ let 
$\mathcal{A}^{\R}_m(Z)$ denote the space 
of all truncations up to order $m$ of real analytic arcs contained in $(Z,0)$. For a complex analytic germ $(Z,0)$, $\mathcal{A}^{\C}_m(Z)$ denotes the space of all truncations up to order $m$ of complex analytic arcs contained in $(Z,0)$.

\begin{thm}\label{MainCor}
Let $\K=\R$ or $\C$.  Let $(X,0)\subset(\K^n,0)$ be a germ of real (resp. complex) analytic set. 
Then  for every $m\in\mathbf{N}$, there is an algebraic germ $(V_m,0)\subset (\K^n,0)$ with the same embedded topological type as $(X,0)$ such that
$\mathcal{A}^{\R}_m(X)=\mathcal{A}^{\R}_m(V_m)$ if $\K=\R$, and  $\mathcal{A}^{\C}_m(X)=\mathcal{A}^{\C}_m(V_m)$, $\mathcal{A}^{\R}_m(X)=\mathcal{A}^{\R}_m(V_m)$ if $\K=\C$.
\end{thm}

We also present an example showing that it is not enough in general to choose for $V_m$ a germ of algebraic set which has a high  tangency with $X$.

\begin{convention*}
 The constants denoted  $C,c$  may change from line to line.   We also use auxiliary constants $C_1, 
 C_2, etc.$.  They also are not fixed and may depend from line to line.  
 Sometimes we write  $C(m)$ to stress that this particular constant may depend on $m$. 
\end{convention*}

\section{Preliminaries}\label{preliminaries}
\subsection{Nash functions and sets}\label{PrelNash}

Let $\K=\C$ or $\mathbb{R}$. Let $\Omega$ be an open subset of
$\K^q$ and let $f$ be an analytic function on $\Omega$. We say that $f$ is a \emph{Nash
function} at $\zeta\in\Omega$ if there exist an open neighbourhood $U$ of $\zeta$ in $\Omega$
and a non zero polynomial $P\in\K[Z_1,\dots,Z_q,W]$ such that $P(z,f(z))=0$ for $z\in U$.  An 
analytic function on $\Omega$ is a Nash function if it is a Nash function at every point
of $\Omega$. An analytic mapping $\varphi:\Omega\to\K^N$ is a \emph{Nash mapping} if each of
its components is a Nash function on $\Omega$.

A subset $X$ of $\Omega$ is called a \emph{Nash subset} of $\Omega$ if for every
$\zeta\in\Omega$ there exist an open neighbourhood $U$ of $\zeta$ in $\Omega$ and Nash
functions $f_1,\dots,f_s$ on $U$, such that $X\cap U=\{z\in U \ ; \     f_1(z)=\dots=f_s(z)=0\}$. A
germ $X_\zeta$ of a set $X$ at $\zeta\in\Omega$ is a \emph{Nash germ} if there exists an open
neighbourhood $U$ of $\zeta$ in $\Omega$ such that $X\cap U$ is a Nash subset of $U$.
Note that $X_\zeta$ is a Nash germ if its defining ideal can be generated by convergent
power series algebraic over the ring $\K[Z_1,\dots,Z_q]$. For more details on real and complex Nash functions and sets see \cite{BCR}, \cite{Tw}.

%%%%%%%%%%%%%%%%%%%%%%%%%%%%%%%%%%%%%%%%%%%%%%%%%%%%%%

\subsection{Nested Artin-P\l oski-Popescu Approximation Theorem}\label{nestedploski}

We set $x=(x_1,\ldots,x_n)$ and $y=(y_1,\ldots,y_m)$. The ring of convergent power series  in $x_1,\ldots, x_n$ is denoted by $\K \{ x\}$. If $A$ is a commutative ring then the ring of algebraic  power series with coefficients in $A$ is denoted by $A\lg x\rg$.

The following result of \cite{BPR} is a generalization of P\l oski's result \cite{ploski1974}. It is a corollary 
of Theorem 11.4 \cite{Sp} which itself is a corollary of the N\'eron-Popescu Desingularization (see \cite{Po}, \cite{Sp} or \cite{Sw} for the proof of this desingularization theorem in whole generality or \cite{Qu} for a proof in characteristic zero).

\begin{thm}\label{nest_ploski}\cite{BPR}
Let $f(x,y)\in \K \lg x\rg[y]^p$ and let us consider a solution $y(x)\in\K \{x\}^m$ of $$f(x,y(x))=0.$$
 Let us assume that $y_i(x)$ depends only on $(x_1,\ldots, x_{\si(i)})$ where $i\lgm \si(i)$ is an increasing function. Then there exist a new set of variables $z=(z_1,\ldots, z_s)$, an increasing function $\t$, convergent power series $z_i(x)\in\K \{x\}$ vanishing at 0 such that $z_1(x),\ldots, z_{\t(i)}(x)$ depend only on $(x_1,\ldots, x_{\si(i)})$, and  an  algebraic power series vector solution $y(x,z)\in\K \lg x,z\rg^m$ of
 $$f(x,y(x,z))=0,$$ such that for every $i$, 
 $$y_i(x,z)\in\K \lg x_1,\ldots,x_{\si(i)},z_1,\ldots,z_{\t(i)}\rg \text{ and }y(x)=y(x,z(x)).$$
%Moreover, given $k\in \N$ we may always require that $y_i(x,z)-y_i(x) \in \ma ^k$, where $\ma =(x_1, ... , 
%x_n) \subset \K\{x,z\}$.  
\end{thm}

%\begin{proof}
%Except the latter claim this theorem is proven in  \cite{BPR}.  This latter claim can be shown by 
%the following standard argument.  We consider only the case $k=1$.  
%Write $y(x) = y(0) + \sum_{j=1}^n x_j w_{j} (x)$.  Then $w(x)$ is a solution of 
%$$g(x,w(x))=  f(x, y(0) + \sum_{j=1}^n x_j w_{j} (x))=0.$$ 
%Let $w(x,z)$ be an algebraic power series vector solution of the above equation, given by the standard version of the theorem,  with convergent power series vector $z(x)$ so that 
%$w(x) = w(x,z(x))$.
%Then $y(x,z): = y(0) + \sum_{j=1}^n x_j w_{j} (x,z)$ solves $f(x,y(x,z))=0$ and satisfies the required properties.
%\end{proof}

%%%%%%%%%%%%%%%%%%%%%%%%%%%%%%%%%%%%%%%%%%%%%%%%%%%%%%%%%%%%%%%%%%%%%%%%%%%%%%%%%%%%%%%%%%%%%%%%%%%%%%%%%%%%%%%%%%%%%%%%%%%%%%%%%%%%%%%%%%%%%%%%%
\subsection{Arc-analytic maps and arc-wise analytic families}
Let $Z, Y$ be analytic spaces. A map $f:Z\rightarrow Y$ is called \emph{arc-analytic} if $f\circ\delta$ is 
real analytic for every real analytic arc $\delta:I\rightarrow Z,$ where $I=(-1, 1)\subset\mathbb{R}$ (cf. \cite{KKurdyka}). 
By \emph{an arc-analytic homeomorphism} we mean a homeomorphism $\varphi$ such that both $\varphi$ and 
$\varphi^{-1}$ are arc-analytic.  

Let $T$ be a nonsingular analytic space. We say that a map 
$f(t, z):T\times Z\rightarrow Y$ is \emph{an arc-wise analytic family in $t$} if it is analytic in $t$ and arc-analytic in $z$. This means 
that  for every real analytic arc $z(s):I\rightarrow Z,$ the map $f(t, z(s))$ is real analytic
and moreover, if $\K = \C,$ complex analytic with respect to $t$ (cf. \cite{parusinski-paunescu}).

%%%%%%%%%%%%%%%%%%%%%%%%%%%%%%%%%%%%%%%%%%%%%%%%%%%%%%%%%%%%%%%%%%%%%%%%%

\subsection{Algebraic Equisingularity of Zariski}\label{varchenko}
Notation:  Let $x=(x_1, \ldots, x_n) \in \K^n$. Then we set $x^i = (x_1, \ldots, x_i) \in \K^i$.

%%%%%%%%%%%%%%%%%%%%%%%%%%%%%%%%%%%%%%%%%%%%%%%%%%%%%%
\subsubsection{Assumptions} \label{pseudopolynomials}
Let $V$ be an analytic hypersurface in a neighborhood of the origin in $\K^l \times  \K^n$ and let $T= V\cap (\K ^l \times  \{0\} )$.  
 Suppose there are given pseudopolynomials
 \footnote{A pseudopolynomial is a polynomial in $x_i$ with coefficients 
 that are analytic in the other variables.
The pseudopolynomials $F_i$ that we consider are moreover distinguished polynomials in $x$, i.e. are  of the form $\displaystyle x_i^p+\sum_{j=1}^pa_j(x^{i-1})x_i^{p-j}$ with $a_j(0)=0$ for all $j$.  They may depend analytically on $t$ that is considered as a parameter.}
 \begin{align*}%\label{pseudopolynomiais:}
F_{i} (t, x^i )= x_i^{p_i}+ \sum_{j=1}^{p_i} a_{i-1,j} (t,x^{i-1})
 x_i^{p_i-j}, \qquad    i=0, ... ,n,
\end{align*}
$t\in \C ^ l$, $x^i \in \C ^i$, 
with analytic coefficients $a_{i-1,j}$, that satisfy 
\begin{enumerate}
\item 
$V=F_n^ {-1} (0)$,
\item
for all $i,j$, $a_{i,j}(t,0)\equiv 0$,
\item
$F_{i-1} $  vanishes on the set of those $(t, x^{i-1})$ for which $F_{i} (t, x^{i-1}, x_i)=0$, considered as an equation on $x_i$, has fewer complex roots than for generic $(t, x^{i-1})$, 
\item
Either $F_i(t,0)\equiv 0$ or $F_i \equiv 1$ (and in the latter case 
$F_k \equiv 1$ for all $k\le i$ by convention),
\item
$F_0\equiv 1$.
\end{enumerate}

For condition (3) it suffices that  the discriminant of $(F_{i})_{red}$, that is of $F_{i}$ reduced,  
divides $F_{i-1}$.  
Equivalently, the first not identically equal to zero  generalized discriminant of $F_{i}$ divides $F_{i-1}$, 
see subsection \ref{generalizeddiscriminants}.  

We call a system $\{F_i(t,x^i)\}$ satisfying the above conditions   \emph{Zariski equisingular}.  
We also say that the family of germs  $V_t := V\cap (\{t\}\times \K^n)$ is  \emph{Zariski equisingular} if such a system with $V=\{F_n=0\}$ exists, maybe after a local analytic change of coordinates $x$.  
Answering a question posed by O. Zariski,  A. Varchenko  \cite{varchenko1972,varchenko1973, varchenkoICM} 
showed that a Zariski equisingular family $V_t$ is locally topologically trivial. The following stronger version of Varchenko's Theorem follows from  Theorems 3.3 and 4.3 of \cite{parusinski-paunescu}.

\begin{thm}\label{theorem} {\rm (\cite{varchenko1972,varchenko1973, varchenkoICM}, \cite{parusinski-paunescu})}
If a system $\{F_i(t,x^i)\}$ is Zariski equisingular then the family $V_t$ is locally topologically trivial along $T$, that is  there exist $\e, \d>0$, $C,c>0$ and a homeomorphism
$$\Phi:B_\e\times \Omega_0\longrightarrow \Omega,$$
where $B_\e=\{t\in\K^l\  ;   \ \|t\|<\e\}$,  $ \Omega_0 = \{x\in \K^n \ ; \    \|x\|<\delta \}$
 and $\Omega$ is a  neighborhood of the origin in $\K^{l+n}$, that preserves the family $V_t$: 
 $\Phi^{-1}(V) = B_\varepsilon \times V_0$.  Moreover, $\Phi$ has a triangular form
 $$\Phi(t,x)=(t,\Psi(t,x))=(t,\Psi_1(t,x^1),\Psi_2(t,x^2),\cdots,\Psi_n(t,x^n))$$
 and satisfies the following additional properties: 
 \begin{enumerate}
 \item
 if $F_n$ is a product of pseudopolynomials then $\Phi$ preserves the zero set of each factor,
 \item
$\Phi$ is subanalytic (semi-algebraic if all $F_i $ are polynomials or Nash functions),
\item
$\Phi$ is an arc-wise analytic family in $t$ and $\Phi^{-1}$ is arc-analytic, 
\item
there is an analytic subset $Z\subset \Omega_0$ such that $\Ph$ is real analytic in the complement 
of $B_\varepsilon \times Z$ and such that  the jacobian determinant of $\Ph$ satisfies for 
$(t,x) \in  B_\varepsilon \times (\Omega_0 \setminus Z)$ 
$$ c \le |jacdet (\Ph) (t,x)| \le C .$$ 
 \end{enumerate}
 Moreover, if the multiplicity of $F_i(0,x^i)$ at $0\in\K^i$ is equal to $p_i$ for every $2\leq i\leq n$ then 
 for all $(t,x)\in B_\e\times \Omega_0$
 $$\left\|\frac{\partial \Psi}{\partial t}(t,x)\right\| \leq C\|\Psi(t,x)\|.$$
\end{thm}

We call the map $\Ph$ satisfying the conditions (1)-(4) \emph{an arc-wise analytic trivialization}. 

\begin{rem}
The last condition can be written equivalently, maybe for different $\e, \d>0$, $C,c>0$,  as 
 $$c \|x\|\leq \|\Psi(t,x)\| \leq C\|x\|, $$
see 
Propositions 1.6 and 1.7  of \cite{parusinski-paunescu}. 
That means geometrically that the trivialization $\Ph$ preserves the size of the distance to the origin. 
\end{rem}

\begin{rem}
The condition (1) of the conclusion  can be given much stronger form, cf. Proposition 1.9 of 
\cite{parusinski-paunescu}.  For any analytic function $G$ 
dividing $F_n$, there are constants $C,c>0$, depending on $G$,  such that 
 $$c |G(0,x)|\leq |G(\Ph(t,x))| \leq C|G(0,x)| .$$
\end{rem}

%\textcolor{red}{Comment : 
%I have changed the constant in the formula from $C_1$ to $C$.  Before there was 
%"there is a constant $C>0$" and $C_1$ in the formula .}

%Here by $\Phi$ being arc-wise analytic in $t$ we mean that for every 
%real analtic arc $x(s):(-1,1)\to \Omega_0$ the map $(t,s) \to \Phi (t,x(s))$ is analytic, cf. \cite{parusinski-paunescu}.
%  By $\Phi^{-1}$ being arc-analytic 
%we mean that for every real analtic arc $(t(s),x(s)):(-1,1)\to \Omega$, the map $s\to  \Phi^{-1} (t(s),x(s))$ is analytic.   

%%%%%%%%%%%%%%%%%%%%%%%%%%%%%%%%%%%%%%%%%%%%%%%%%%%%%%
 
%%%%%%%%%%%%%%%%%%%%%%%%%%%%%%%%%%%%%%%%%%%%%%%%%%%%%%

\section{Nash approximation of analytic sets}\label{NashSect}

\subsubsection{Generalized discriminants}\label{generalizeddiscriminants}(see e.g. \cite{whitneybook} Appendix IV)
Let $f(T) = T^p + \sum_{j=1}^ p a_iT^{p-i} = \prod_{j=1}^ p (T-T_i)$.  Then the expressions 
$$
 \sum_{r_1, \ldots r_{j-1}} \,  \prod_{k< l, k,l \ne r_1, \ldots r_{j-1}}  (T_k-T_l)^2
$$
are symmetric in $T_1, \ldots , T_p$ and hence polynomials in $a=(a_1, \ldots, a_{p}) $.  We denote 
these polynomials by $\Delta_j (a)$.  
Thus $\Delta_1$ is the standard discriminant and $f$ has exactly $p-j$ distinct roots if and only if 
$\Delta_1=\cdots = \Delta _{j}=0$ and $\Delta_{j+1}\neq 0$.

%%%%%%%%%%%%%%%%%%%%%%%%%%%%%%%%%%%%%%

\subsubsection{Construction of a normal system of equations}\label{NormalSystem}
 We recall here the main construction of \cite{mostowski}, \cite{BPR}.  
 
 Let $g_1, \ldots, g_l \in\K\{x\}$ be a finite set of  pseudopolynomials:  
   \begin{align*}%\label{polynomiai:f_n}
g_{s} ( x)= x_n^{r_s}+ \sum_{j=1}^{r_s} a_{n-1,s,j} (x^{n-1}) x_n^{r_s-j}.
\end{align*}
We arrange the coefficients $a_{n-1,s,j}$ in a row vector  
$a_{n-1} \in \K\{x^{n-1}\}^{p_n}$, $p_n=\sum_s r_s$.  
Let $f_n$ be the product of the $g_s$'s.  The generalized discriminants $\Delta_{n,i} $ of $f_n$ are 
polynomials in $a_{n-1}$.    Let $j_n$ be a positive integer such that  
 \begin{align}\label{discriminants:n}
 \Delta_{n,j_n}  ( a_{n-1} ) \not \equiv 0\quad  \text { and } \quad 
\Delta_{n,i} ( a_{n-1} )\equiv 0  \text { for } i<j_n   . 
\end{align}
 Then, after a linear change of coordinates $x^{n-1}$, we may write 
    \begin{align*}%\label{polynomiai:f_n-1}
 \Delta_{n,j_n} ( a_{n-1} ) =  u_{n-1} (x^{n-1})\left( x_{n-1}^{p_{n-1}}
 + \sum_{j=1}^{p_{n-1}} a_{n-2,j} (x^{n-2}) x_{n-1}^{p_{n-1}-j} \right) 
\end{align*}
where $u_{n-1}(0)\ne 0$ and for all $j$, $a_{n-2,j}(0)=0$,  and 
$$
f_{n-1} = x_{n-1}^{p_{n-1}}+
  \sum_{j=1}^{p_{n-1}} a_{n-2,j} ( x^{n-2}) x_{n-1}^{p_{n-1}-j} $$
is the Weierstrass polynomial associated to   $ \Delta_{n,j_n}$. We denote  
by $a_{n-2} \in \K\{x^{n-2}\}^{ p_{n-1}}$  the vector 
of its coefficients $a_{n-2,j}$.   
%Let $j_{n-1}$ be the positive integer such that the first $j_{n-1}-1$ generalized discriminants $\Delta_{n-1,i} $ 
%of $f_{n-1}$ are identically zero and  $\Delta_{n-1,j_{n-1}} $ is not.  Then again we define 
%$f_{n-2} ( x^{n-2})$ as the Weierstrass polynomial associated to  $\Delta_{n-1,j_{n-1}}  $. 
 
Similarly we define recursively a sequence of pseudopolynomials $f_{i} (  x^i )$, $i=1, \ldots, n-1$, such that 
 $f_i= x_i^{p_i}+ \sum_{j=1}^{p_i} a_{i-1,j} (x^{i-1}) x_i^{p_i-j}  $ is the Weierstrass polynomial associated to the first 
 non identically equal to zero generalized discriminant $\Delta_{i+1,j_{i+1}} ( a_{i} )$ of $f_{i+1}$, 
where  we denote in general $a_{i}= (a_{i,1} , \ldots , a_{i,p_{i+1}} )$ and 
  \begin{align}\label{polynomials:f_i}
 \Delta_{i+1,j_{i+1}} ( a_{i} ) =  u_{i} (x^{i})  \left(x_i^{p_i}+ \sum_{j=1}^{p_i} a_{i-1,j} (x^{i-1}) x_i^{p_i-j}  \right) ,  
 \quad i=0,...,n-1.  
\end{align}
 Thus the vector  
of functions $a_i$ satisfies  
  \begin{align}\label{discriminants:i}
\Delta_{i+1,k} ( a_{i} )\equiv 0 \qquad k<j_{i+1}   ,  \quad i=0,...,n-1. 
\end{align}
This means in particular that 
  \begin{align*}%\label{discriminants:0}
\Delta_{1,k} ( a_{0} ) \equiv 0 \quad \text {for } k<j_1  \text { and }  \Delta_{1,j_1} ( a_{0} ) \equiv u_0 ,
\end{align*}
where $u_0$ is a non-zero constant.  

\begin{rem}\label{mult}
At each step of this construction we may use a linear change of coordinates in order to assume that the multiplicity of $f_i$ at the origin is equal to $p_i$, that is $f_i$ as a power series does not contain  monomials of degree smaller than $p_i$.  Equivalently it means that for all $j$, $\mult_0 a_{i-1,j}\ge j$.   We will assume  in the following that this condition is satisfied for every $i$.  
\end{rem}

%%%%%%%%%%%%%%%%%%%%%%%%%%%%%%%%%%%%%%%%%%%%%%

\subsubsection{Approximation by Nash functions}\label{ApproxNashFunct} 
If we consider \eqref{polynomials:f_i} and 
\eqref{discriminants:i} as a system of polynomial equations on $a_i(x^i)$, $u_i (x^i)$ then, by construction,  
this system admits convergent solutions.  Therefore, by Theorem \ref{nest_ploski},   there exist a new set of variables $z=(z_1,\ldots, z_l)$, an increasing function $\tau$,  convergent power series $z_i(x)\in\K \{x\}$ {vanishing at 0}, algebraic power series $u_i (x^i,z)  \in\K \langle x^{i},z_1,\ldots,z_{\tau(i)}\rangle$ and  vectors of algebraic power series 
$a_i(x^i,z) \in\K \langle x^{(i)},z_1,\ldots,z_{\tau(i)}\rangle^{p_{i+1}}$
 such that the following holds:
\begin{align*}
& z_1(x),\ldots, \,z_{\tau(i)}(x) \text{ depend only on }(x_1,\ldots, x_{i}),\\
& a_i(x^i,z),\ u_i (x^i,z) \text{ are solutions of  }\eqref{polynomials:f_i},  
\eqref{discriminants:i}, \\
& a_i(0,z)\equiv 0,\\
& \text{and }a_i(x^i)= a_i(x^i,z(x^i)),\   u_i(x^i)= u_i(x^i,z(x^i)).\\
\end{align*}
Then we define 
\begin{align*}%\label{polynomiaisdef:F_n(t)}
 F_n(z,x) = \prod_s G_s(z,x)  , \quad G_{s} (z, x)=  x_n^{r_s}+ 
 \sum_{j=1}^{r_s} a_{n-1,s,j} \left(x^{n-1}, z(x^{n-1})-z\right) x_n^{r_s-j},  
\end{align*}
 \begin{align*}%\label{polynomiaisdef:F_i(t)}
 F_i(z,x) = x_i^{p_i}+ \sum_{j=1}^{p_i} a_{i-1,j} \left(x^{i-1}, z^{\t(i-1)}(x^{i-1})- z^{\t(i-1)}\right) x_i^{p_i-j} , \quad i=0, \ldots , n-1.
\end{align*}
Finally we set $F_0\equiv 1$.  

We consider the system $\{F_{i} (z, x)\}$ as a system of pseudopolynomials in $x$ 
parameterized by $z$.  If we substitute $z=0$ we obtain the original normal system of equations.  
The system $\{F_{i} (z, x)\}$ is  Zariski equisingular.  Indeed,  $a_{i,j} (0,z) \equiv 0$  and hence the condition (2) of Assumptions 2.4.1 is satisfied.  Moreover for every $F_i$, $F_{i-1}$ is the first not identically equal to zero generalized discriminant, and hence the condition (3) of
 Assumptions 2.4.1 is satisfied.   
 Therefore Theorem \ref{theorem} implies  the following lemma. 

\begin{lem}\label{lemma}
There are $\e, \d>0$, $C,c>0$ and a subanalytic, arc-wise analytic as a family in $z$,   an arc-analytic  homeomorphism
\begin{align}
\Ph (z,x) = (z, \Ps (z,x)) : B_{\varepsilon} \times \Omega_0 \to \Omega, 
\end{align}
where $B_{\varepsilon} = \{z\in \K^l \ ; \    \|z\|<\varepsilon\}$, $ \Omega_0 = \{x\in \K^n \ ; \    \|x\|<\delta \}$,
 and $\Omega$ is a  neighborhood of the origin in $\K^{l+n}$, 
  preserving the sets $\{G_s=0\}$ for $1\leq s\leq l$, i.e. $\Phi ^{-1} (\{G_s=0\}) = B_\varepsilon \times (\{G_s=0,z=0\})$.  
There is an analytic nowhere dense subset $Z\subset \Omega_0$ such that $\Ph$ is real analytic in the complement 
of $B_\varepsilon \times Z$ and  such that  the jacobian determinant of $\Ph$ satisfies for $(z,x) \in  B_\varepsilon \times (\Omega_0\setminus Z)$ 
\begin{align}\label{jacobian}
c\le |jacdet (\Ph) (z,x)| \le C .
 \end{align}
 Moreover, there exists a constant $C_1>1$ such that for all $(z,x)\in B_\e\times \Omega_0$
 \begin{align}\label{boundedder}
&  \left\|\frac{\partial \Psi}{\partial z}(z,x)\right\| \leq C_1\|\Psi(z,x)\|  \\ \notag
&  C_1^{-1} \|x\| \leq \|\Psi(z,x)\| \leq C_1\|x\|
 .\end{align}
 \end{lem}

Let us consider for $m\in \N$ 
$$
z(x) = z_m(x) + \tilde z_m (x),
$$
where $ z_m(x)$ is the $m$-th Taylor polynomial of $z(x)$.  %Then $\| \tilde z_m(x) \| = O(\|x\|^{m+1})$. 
We shall use in the sequel the following bounds that hold for $x$ sufficiently close to the origin.
\begin{align}\label{taylorbounds}
 \|\tilde z_m(x)\| \le C(m)  \|x\|^{m+1}, \qquad  \|\frac {\partial \tilde z_m}{\partial x} (x)\| \le C(m)  \|x\|^{m} .
 \end{align} 
By  classical formulae we may bound $C(m) \le C \sup_{\|x\|<(3/2)\delta '} \|z_\C (x)\| \delta'^{-(m+1)}$, where  
 $z_\C$ is  the complexification of $z$ (if $\K=\R$), that we suppose defined on $\|x\|<2\delta '  $, $0<\delta'< (1/2) \delta$, and  $C$ is a universal constant. Then \eqref{taylorbounds}  holds on  $\|x\|\le \delta '  $.

Set  
\begin{align*}%\label{polynomiaisdef:F_n(t)}
 \tilde f_n(t,x) & = \prod_s \tilde g_s(t,x)  , \quad \tilde g_{s} (t, x)= G_{s} (t\tilde z_m(x), x) \\
\tilde  f_i(t,x) & =  F_i( t\tilde z_m(x), x) , \quad i=0,\ldots, n-1.
\end{align*}

The system $\{\tilde f_{i} (t, x)\}$ is a deformation of  the original normal system of equations, given by  $t=0$,   
to a Nash system of equations, given by $t=1$.  It is Zariski equisingular by construction.  
We show that its arc-wise analytic trivialization can be induced from  $ \Phi$ given by Lemma \ref{lemma}.

%\textcolor{red}{Comment: The formula for $\varphi$ has been corrected.  The proof below has been changed completely. }
The construction of $\Ph$ of Theorem \ref{theorem} given in \cite{parusinski-paunescu} is  universal, as we explain below, 
 provided the Whitney Interpolation map, $\ps$ in the notation of \cite{parusinski-paunescu}, is fixed.  
Thus the map  of Lemma \ref{lemma}  is of the form
$$
\Ph (z,x) = (z, \Ps (z,x))= (z, \Ps_1 (z,x^1),  \Ps_2 (z ,x^2),  \ldots , \Ps_n (z,x^n)) $$ 
and is constructed inductively as follows.   
Denote for short $x'= x^{n-1}$ and  let 
$$\Ph' (z,x') = (z, \Ps' (z,x'))= 
 (z, \Ps_1 (z,x^1),   \ldots , \Ps_{n-1} (z,x')) .$$
Let  $\eta (z,x') = (\eta_1 (z,x'), ... , 
\eta_{p_n} (z,x')) $ be the vector of the roots of $F_n(z,x',x_n)$ (it is denoted by $a (t,x') $ in \cite{parusinski-paunescu}).  
Then by (3.5) of \cite{parusinski-paunescu}, $\Ps _n $ is defined by 
\begin{align}\label{psn}
\Ps _n (z,x', x_n ) =  \ps (x_n, \eta (0,x'), \eta ( z, \Ps' (z,x'))) , 
\end{align} 
where $\ps$ is the Whitney Interpolation map, see \cite{parusinski-paunescu} Appendix I.  

In our case, the coefficients, and hence the roots, of  $F_i (z,x^{i-1}, x_i)$, depend only on $z^{\tau_{i-1}}, x^{i-1}$.  
Hence $\Ph$ is of the form   
$$
\Ph (z,x) = (z, \Ps (z,x))= (z, \Ps_1 (x^1),  \Ps_2 (z^{\tau(1)} ,x^2),  \ldots , \Ps_n (z^{\tau(n-1)},x^n)).$$
Now let 
$$
H (t,x)= (t, h (t,x))= (t, h_1 (x^1),  h_2 (t ,x^2),  \ldots , h_n (t,x^n)).$$
 be the arc-wise analytic trivialization constructed by the same recipe for 
the system $\{\tilde f_{i} (t, x)\}$ and let 
$$H' (t,x') =(t, h' (t,x' ))= (t, h_1 (x^1),  h_2 (t ,x^2),  \ldots , h_{n-1} (t,x^{n-1}))$$
 be the trivialization of  $\{\tilde f_{i} (t, x')\}_{i\le n-1}$ obtained in the inductive step.  
Then $h$ can be deduced from $\Ps$ by the following recursive formula. 

\begin{lem}\label{lemma2}
\begin{align}\label{varphi}
h (t,x) =   \Ps ( t\tilde z_m( h' (t,x')), x).
\end{align}
\end{lem}

\begin{proof} 
The above formula means that for each $i=1, \ldots, n$
\begin{align}\label{varphi_i}
h_i (t,x^{i}) =    \Ps_i ( t\tilde z^{\tau(i-1)}_m( h^{i-1} (t,x^{i-1})), x^{i}).
\end{align}
In particular, $\Ps_1$ is independent of $z$, $h_1$ is independent of $t$ and $h_1 (x_1) = \Ps_1(x_1)$.  

We assume by induction that $h' (t,x') =   \Ps' ( t\tilde z_m^{\tau(n-2)}
( h^{n-2} (t,x^{n-2})), x' )$ and then show that 
\eqref{varphi_i} holds for $i=n$.
\begin{align*}%\label{pshn}
& \Ps _n ( t\tilde z_m( h' (t,x')), x', x_n) \\ 
& =  \ps (x_n, \eta (0,x'), \eta (  t\tilde z_m( h' (t,x')), \Ps' ( t\tilde z_m( h' (t,x')),x')))  \\
& =  \ps (x_n, \eta (0,x'), \eta (  t\tilde z^{\tau(n-1)}_m( h' (t,x')), \Ps' ( t\tilde z^{\tau(n-2)}_m( h^{n-2} (t,x^{n-2})), x' )))  \\
& =  \ps (x_n, \eta (0,x'), \eta (  t\tilde z^{\tau(n-1)}_m( h' (t,x')), h' (t,x')))  \\
& =  \ps (x_n, \xi (0,x'), \xi ( t, h' (t,x')))  = h_n(t,x),
\end{align*} 
where    $ \xi ( t, h' (t,x'))$ is the vector of the roots of $\tilde f_n  (t, h' (t,x') , x_n) $.  
\end{proof}

% Fix $R>1$.  Then to show the last part it is enough to note that we may choose $\delta'$ so that 
%$$
%\|t\tilde z_m(x)\|< \varepsilon 
%$$
%for $\|x\|<\delta '$,  $|t|<R$,  and all $m\in \N$.  

%%%%%%%%%%%%%%%%%%%%%%%%%%%%%%%%%%%%%%%%%%%%%%%%%%%%%%%

%%Write $\varphi(t,x) = (t, \psi(t,x))$, i.e. $\psi(t,x):=\Psi(t\tilde z_m(x),x)$.  
%%Then there is a constant $C>0$ such that we have for $x$ close to the origin anf $t$ from a neighborhood of $[0,1]$ in $\R$ 
%\begin{align} %\label{boundedder}
%%\left \| \frac {\partial h} {\partial t} (t,x))\right\|  \le  C C(m) \|x\|^{m+2} , 
%\end{align}

In particular, by \eqref{boundedder} and \eqref{varphi}
 \begin{align}\label{boundh}
 \|h(t,x) \|\le C_1 \|x\|.
 \end{align}

There is a constant $C>0$ such that we have for $x$ close to the origin and $t$ from a neighborhood of $[0,1]$ in $\R$ 
\begin{align}
\left \| \frac {\partial h} {\partial t} (t,x)\right\|  \le  C C(m) \|x\|^{m+2} , 
\end{align}
where $C(m)$ is given by \eqref {taylorbounds}.  Indeed, this is obvious for $h_1$ since $h_1$ is independent of $t$.  
The inductive step then follows from  the identity \eqref{varphi} and the inequalities \eqref{boundedder}, \eqref {taylorbounds} and \eqref{boundh}:
\begin{align*}%\label{boundedder}
\left \| \frac {\partial h} {\partial t} (t,x)\right\| 
 & \le  \| \frac {\partial \Ps } {\partial z}  (t\tilde z_m(h'(t,x')),x)\|  \times  \\ \notag 
&  \times   ( \|\tilde z_m(h'(t,x'))\|  
  +  \|t \frac{\partial \tilde z_m } {\partial x} 
 (h'(t,x'))\| 
\left \|  {\partial h'} /{\partial t} (t,x')\right\|  ) \\ \notag
& \le  C'  C(m) \|\Ps (t\tilde z_m(h'(t,x')),x)\|  \|x\|^{m+1} \\ \notag
& \le  C C(m) \|x\|^{m+2} .  
\end{align*} 
Then by integration we obtain
\begin{align}\label{boundedder2}
 \| h (1,x) - x \|   \le  C C(m) \|x\|^{m+2}.
\end{align}
This bound is crucial for showing  the following result.

\begin{prop}\label{Prop1}
Let $X$ be an analytic subset of an open neighborhood $U_0$ of 
$0\in\K^n$ with $0\in X.$ Then  there are constants $C,c>0$ and an open neighborhood $U\subset U_0$
of $0$ in $\K^n$ such that for every $m\in\mathbf{N}$, %$m\ge 1$, 
there are a subanalytic  arc-analytic homeomorphism $\phi_m:U\rightarrow\phi_m(U)\subset\K^n$ and
a Nash subset $V_m$ of $\phi_m(U)$
with the following properties:
\vspace*{2mm}\\
(a)  $\phi_m(X\cap U)=V_m \cap \phi_m(U),$\\
(b) $\|\phi_m(a)-a\|\leq C^m\| a\|^m$ for every $a\in U.$ 
\vspace*{2mm}\\
Moreover,  there is a nonwhere dense  analytic subset $Z\subset U$ such that $\ph_m$ is real analytic in the complement of $Z$  and 
$$ c\le |jacdet (\varphi_m) (x) | \le C ,$$
for $x \in  U\setminus Z$.
\end{prop}

\begin{proof}
We may assume that $X$ is defined by the pseudopolynomials  
$g_1, \ldots, g_l \in\K\{x\}$. 
We set $\hat g_i(x) = G_i (1,x)$ and we denote by $V_m$ the zero locus of the $\hat g_i$. 
Then, by Lemma \ref{lemma2}, $X\cap U$ is homeomorphic to $V_m\cap \phi_m (U)$ where $U=\{x\in\K^n \ ; \  \|x\|<\delta'\}$ 
and the homeomorphism $\phi_m (x)$  is obtained  by  setting $t=1$ in \eqref{varphi} 
$$
\phi_m (x) = \Ps ( \tilde z_{m-2} (x), x).$$
(We may assume that $m>1$.) Condition (b) now follows from \eqref{boundedder2}.  
 
The last claim of the proposition follows from \eqref{jacobian}.  Indeed,  $\Ps_i $ depends only on $z_1, ..., z_{\tau_{i-1}}$ and  $x^{i}$  and therefore  
%$$
%\frac {\partial } {\partial y_i} \Ps_j (\tilde z_m(y), y) = 
% \frac {\partial \tilde z_m } {\partial y_i} (y) \frac {\partial  \Ps_j} {\partial z} ( \tilde z_m(y), y)  + 
%\frac {\partial  \Ps_j } {\partial x_i}( \tilde z_m(y), y) 
%$$
$$\frac{\partial}{\partial y_i}\Psi_j(\tilde z_m(y),y)=\sum_{k=1}^s\frac{\partial \tilde z_{m,k}}{\partial y_i}(y)\frac{\partial \Psi_j}{\partial z_k}(\tilde z_m(y),y)+\frac{\partial \Psi_j}{\partial x_i}(\tilde z_m(y),y))$$ 
is non-zero only for $j\ge i$.  (Formally the above formula makes sense only for $\K=\R$.  In the complex case we should consider the derivatives with respect to $\re y_i$ and $\im y_i$.)  Moreover, for $i=j$ we have $\frac {\partial } {\partial y_i} \Ps_j ( \tilde z_m(y), y) = \frac {\partial  \Ps_j} {\partial x_i} ( \tilde z_m(y), y) $, since $z_1(x), ..., z_{\tau_{i-1}}(x)$ 
depend only on $x^{i-1}$.  Therefore, taking into account that the jacobian matrix of $\Ph (z,x)$ is 
block triangular, 
$$jacdet (\varphi_m) (y) = jacdet (\Ps (\tilde z_m(y), y)) = jacdet (\Ph)  (\tilde z_m(y), y), $$
and the claim now follows from \eqref{jacobian}. 
\end{proof}

%%%%%%%%%%%%%%%%%%%%%%%%%%%%%%%%%%%%%%%%%%%%%%

%%%%%%%%%%%%%%%%%%%%%%%%%%%%%%%%%%%%%%%%%%%%%%%%%%%%%%%%%%%%%%%%%%%%%%%%%%%%%%%
\section{Algebraic approximation of Nash sets}\label{AlgSect}
The main result of this section is the following proposition. For any $\K$-differ\-entiable map $\phi:U\subset\K^n\rightarrow\K^n,$ let $J_{\phi}$ denote the map assigning to every $a\in U$ the jacobian matrix of $\phi$ at $a$.
\begin{prop}\label{Prop2}
Let $X$ be a Nash subset of an open neighborhood $U_0$ of 
$0\in\K^n$ with $0\in X.$ Then  
there are a constant $C>0$ and an open neighborhood $U\subset U_0$
of $0$ in $\K^n$ such that for every $m\in\mathbf{N}$
there are a Nash  diffeomorphism $\phi_m:U\rightarrow\phi_m(U)\subset\K^n$ and
an algebraic subset $V_m$ of $\K^n$ 
with the following properties:
\vspace*{2mm}\\
(a)  $\phi_m(X\cap U)=V_m\cap\phi_m (U),$\\
(b) $||\phi_m(a)-a||\leq C^m||a||^m$ for every $a\in U.$
\vspace*{2mm}\\
Moreover, the sequence $(J_{\phi_m})_{m\in\mathbf{N}}$ converges uniformly on $U$ to the constant map assigning to every $a\in U$ the identity matrix.
\end{prop}

\begin{rem} By the last assertion, for every $\varepsilon>0$ the maps $\phi_m$ in Proposition \ref{Prop2} can be 
chosen in such a way that $1-\varepsilon<|jacdet(\phi_m)(x)|<1+\varepsilon$ for every $x\in U,$ $m\in\mathbf{N}.$ 
\end{rem}

\noindent\textit{Proof of Proposition \ref{Prop2}.}  First we discuss the case $\K=\C.$ We may assume that $\mathrm{dim}(X)=k<n$ because otherwise there is nothing to prove. By Artin-Mazur's
construction \cite{AM} (cf. also \cite{bochnakkucharz1984}), there are $s\in\mathbf{N}$, an algebraic set $M\subset\C^n\times\C^s$ 
with $0\in M$ and a complex Nash submanifold $N$ of a polydisc $E\times F\subset\C^n\times\C^s$ 
centered at zero such that:\vspace*{2mm}\\
(i) $\pi|_M:M\rightarrow\C^n$ is a proper map, where 
$\pi:\C^n\times\C^s\rightarrow\C^n$ denotes the natural projection, and
$M\cap(E\times F)\subset N,$\\
(ii)
$\pi|_{N}:N\rightarrow E$ is a biholomorphism and
$\pi(M\cap(E\times F))=X\cap E.$\vspace*{2mm}

We may assume that $s=1$ (i.e. $F$ is a disc in $\C$). Indeed, we may replace $M,N$ by $v_L(M), v_L(N)\subset\C^n\times\C,$
respectively.
Here $v_L$ is defined by the formula $v_L(x,z)=(x,L(z)),$
where $L:\C^s\rightarrow\C$ is a linear form such that $v_L$ restricted 
 to $(\{0\}^n\times\C^s)\cap M$ is injective. (Note that it may be necessary
to replace $E\times F$ by a smaller polydisc.)

We may also assume (applying a change of variables) that 
$\pi(M)\subset\C^n=\C^{n-1}\times\C$ has proper projection
onto $\C^{n-1}$ and $X\cap (E'\times E'')$ has proper projection onto $E',$
where $E=E'\times E''\subset\C^{n-1}\times\C.$

Denote $\pi=(\pi_1,\ldots,\pi_{n-1},\pi_n)=(\pi',\pi_n).$ Then after replacing $\pi_n$ by any
polynomial $w$ we obtain a polynomial map $\rho=(\pi',w)$ such that $\rho|_M$ is proper (so
$\rho(M)$ is an algebraic subset of $\C^n$). Now,
by (ii), there is a holomorphic function $\tau:E'\times E''\rightarrow F$ such that $N$
is the graph of $\tau$. We will choose $w=w_m$ such that (perhaps after shrinking $E', E'', F$) for $\rho_m=(\pi',w_m)$
the map $\phi_m(x)=\rho_m((x,\tau(x)))$ satisfies (a) and (b) with $V_m=\rho_m(M)$. 
The size of $E'\times E''\times F$ will be independent of $m.$  The idea to obtain algebraic $V_m$
(equivalent to or approximating some $X$) as the image of an algebraic set by a certain polynomial map $\rho_m$ appeared before (see e.g. \cite{bochnakkucharz1984}, \cite{AK}, \cite{B2}). But here $\rho_m$ must be chosen in a special way to ensure that $V_m$
is both analytically equivalent and higher order tangent to $X.$

Let us verify that after shrinking $E', E'', F$ the following hold:\vspace*{2mm}\\
(x) $(\overline{E'}\times\C\times\partial F)\cap M=\emptyset,$\\
(y) $(\overline{E'}\times\partial E''\times\overline{F})\cap M=\emptyset,$\\  
(z) $(\overline{E}\times\partial F)\cap\overline{N}=\emptyset.$\vspace*{2mm}

First, by the fact that $N\subset E\times F$ is a graph of a holomorphic function defined
on $E,$ we know that (z) holds after replacing $E$ by any of its relatively compact subsets. Hence, we may assume
to have (z). Observe also that if (z)
is true, then it remains such after shrinking $F$ slightly with fixed $E$.

Next, by the facts that $\pi|_M$ is proper and $\pi(M)$ has proper projection onto $\C^{n-1}$, we know
that $(\{0\}^{n-1}\times\C\times\C)\cap M$ is a finite set. Therefore after shrinking $F$ slightly
(in such a way that (z) remains true) we have $(\{0\}^{n-1}\times\C\times\partial F)\cap M=\emptyset.$
Thus, in view of the properness of the projections, (x) holds for $F$ chosen above with any sufficiently small neighborhood
$E'$ of $0.$ 
It follows that both (x) and (z) hold with any sufficiently small
$E=E'\times E''.$ We can pick $E$ in such a way that $(\overline{E'}\times\partial E'')\cap\pi(M)=\emptyset,$
which implies (y).

Let $\delta$ denote the radius of
$F.$ Define $w_m:E\times F\rightarrow\C$ by the formula 
$w_m(x,z)=\pi_n(x,z)+(\frac{z}{\delta})^m$ with large $m$.
Let us check that $\phi_m:U\rightarrow\C^n$ given by $\phi_m(x)=\rho_m((x,\tau(x)))$ 
has all the required properties, where
$U=E'\times E''.$ For $z\in F$ we have $|z|<\delta$ and 
the graph of $\tau$ is contained in $E'\times E''\times F,$ hence in view of (z), for $m$ large enough,
$\phi_m$ is a biholomorphism onto its image. The injectivity of 
$\phi_m$ requires a brief explanation. 
%By the definition of $\phi_m$ the problem of injectivity 
%easily boils down to the problem of injectivity of
%a one variable holomorphic function of the form $g_m(x_n)=x_n+k_m(x_n),$ defined on a disc such that for
%large $m$ the derivative of $k_m(x_n)$ has very small modul. Such $g_m$ clearly is injective.\\
First, by the definition of $\phi_m,$ we can explicitly write 
\begin{equation}\label{phim}\phi_m(x)=\left(x_1,\cdots,x_{n-1},x_n+\left(\frac{\t(x)}{\d}\right)^m\right).
\end{equation}
Thus $\phi_m$ is injective on $U$ if the map $x_n\mapsto x_n+\left(\frac{\t(x)}{\d}\right)^m$ is injective for every fixed $(x_1,\cdots,x_{n-1})$. The latter assertion is true for large $m$ if $\mathrm{sup}_{x\in U}|\frac{\tau(x)}{\delta}|<1$
(because then the modulus of the derivative of the map $x_n\mapsto \left(\frac{\t(x)}{\d}\right)^m$ is small). 
Let us check that $\mathrm{sup}_{x\in U}|\frac{\tau(x)}{\delta}|<1.$ Recall that $N=\mathrm{graph}(\tau).$
Hence, by (z), we have
$$0<\mathrm{inf}\{||a-b||\ ; \ {a\in\overline{\mathrm{graph(\tau)}}, b\in\overline{E}\times\partial F}\}
\leq\mathrm{inf}\{|\tau(x)-c| \ ; \  x\in\overline{E}, c\in\partial F \}.$$ Since $\delta$ is the radius of $F,$ and $U=E,$ we obtain $\mathrm{sup}_{x\in U}|\tau(x)|<\delta,$ as required.

Let us verify that for large $m,$ 
$$\phi_m(X\cap(E'\times E''))=\rho_m(M)\cap\phi_m(E'\times E'').$$ 
By  (i), (ii) and the definition of $\phi_m,$ we have $\phi_m(X\cap(E'\times E''))=\rho_m(\mathrm{graph}(\tau)\cap M)$. Moreover, $\rho_m(\mathrm{graph}(\tau))=\phi_m(E'\times E''),$ hence it is sufficient to
show that $\rho_m(\mathrm{graph}(\tau)\cap M)=\rho_m(\mathrm{graph}(\tau))\cap\rho_m(M).$
In fact, the inclusion "$\subset$" is trivial so we prove "$\supset$".

Fix $a\in\rho_m(\mathrm{graph}(\tau))\cap\rho_m(M).$
Then there are $z\in\mathrm{graph}(\tau)$ and $v\in M$ such that $\rho_m(z)=\rho_m(v)=a.$
Observe that $z,v\in E'\times\C\times\C.$ Indeed, write $z=(z_1,\ldots,z_{n-1},z_n,z_{n+1}),$
$v=(v_1,\ldots,v_{n-1},v_n,v_{n+1}).$ Since $\mathrm{graph}(\t)\subset E'\times E''\times F$ we have
$(z_1,\ldots,z_{n-1})\in E'.$ By the definition of $\rho_m$ and by $\rho_m(z)=\rho_m(v)$ we have $(z_1,\ldots, z_{n-1})=(v_1,\ldots,v_{n-1})$.

Moreover, $(E'\times\C\times\C)\cap M$ is bounded so
$v$ must belong to $E'\times\C\times{F}$ because otherwise (for large $m$)
in view of the definition of $w_m$ and (x),
$\rho_m(v)=a$ lies outside $\rho_m(\mathrm{graph}(\tau)),$ which is a contradiction.
Now if $v\notin E'\times E''\times{F},$ then by (y), (z), (x) (for large $m$) $\rho_m(v)\notin\rho_m(\mathrm{graph}(\tau)),$
again a contradiction. Consequently, in view of (i), we have $v\in M\cap (E'\times E''\times F)\subset\mathrm{graph}(\tau).$ 
Since $\rho_m|_{\mathrm{graph}(\tau)}$ is injective, we have $v=z,$ hence 
$a\in\rho_m(\mathrm{graph}(\tau)\cap M).$ This shows that 
$$\phi_m(X\cap(E'\times E''))=\rho_m(M)\cap\phi_m(E'\times E'')$$ and completes the proof of (a).

Finally, by (\ref{phim}) we have 
$$||\phi_m(a)-a||\leq\frac{|\tau(a)|^m}{\delta^m}$$ for every $a\in E'\times E''.$ Since 
$\tau(0)=0,$ we immediately obtain (b). 

The last assertion of the proposition follows by (b). Namely, we may assume that $U$ is so small that $c||a||<1$ for every $a\in U.$ Then (b) implies that $(\phi_m)_{n\in\mathbf{N}}$ converges to $id$ uniformly on $U.$ In view of the Weierstrass theorem, after shrinking $U$ slightly, we obtain that $(\frac{\partial\phi_m}{\partial x_i})_{m\in\mathbf{N}}$ converges uniformly on $U$ to $\frac{\partial id}{\partial x_i}$ for every $i=1,\ldots,n.$
This completes the proof in the case $\K=\C$.

As for $\K=\mathbb{R},$ we cannot simply repeat the procedure above because
the image of a real algebraic set by a proper polynomial map need not be algebraic. Instead we proceed as
follows. First for the given Nash set $X$ let $X_{\C}$ denote its complexification. More precisely,
$X_{\C}$ is a representative of the smallest complex Nash germ in $(\C^n,0)$ containing the germ $(X,0).$
Note that $X_{\C}$ is defined by real equations. For $X_{\C}$ one can repeat the construction described above
obtaining $M,N$ also defined by real equations. In particular, the Taylor expansion of $\phi_m$ around zero has real coefficients i.e. the restriction of $\phi_m$ to $\mathbb{R}^n$ is a real Nash isomorphism.
It is not difficult to observe that $(\phi_m(X),0)$ is the
germ of a real algebraic set as required.\qed

Theorem \ref{Main} now follows from Proposition \ref{Prop1} and Proposition \ref{Prop2}. 
%%%%%%%%%%%%%%%%%%%%%%%%%%%%%%%%%%%%%%%%%%%%%%%%%%%%%%%%

%%%%%%%%%%%%%%%%%%%%%%%%%%%%%%%%%%%%%%%%%%%%%%%%%%%%%%%%

\section{Proof of Theorem \ref{MainCor}}\label{section_cor}
Let $(X,0)\subset (\K^n,0)$ be a germ of analytic set. In both real and complex cases $(X,0)$ can be considered as a real analytic set germ defined by equations
$$g_1=\cdots=g_l=0$$ 
where the $g_i$ are convergent power series with real coefficients. A real analytic arc on $(X,0)$ is a germ of real analytic map $(\R,0)\lgw (X,0)$. Then the space of real analytic arcs of $(X,0)$ is in bijection with set of morphisms
$$\frac{\R\{x_1,\ldots,x_n\}}{(g_1,\ldots,g_l)}\lgw \R\{t\}$$
which are defined by the data of $n$ convergent power series $x_1(t),\ldots, x_n(t)\in t\R\{t\}$ such that
$$g_i(x_1(t),\ldots,x_n(t))=0\ \ \ \forall i.$$

Let $m$ be a positive integer and let $\phi_m$ be the homeomorphism given by Theorem \ref{Main}. 
Let $\gamma$ be an arc on $(X,0)$, i.e. $\gamma$ is a real analytic map $(-\e,\e)\lgw X\cap U$ for some $\e>0$.  
Then we define $\gamma':(-1,1)\lgw X\cap U$ by $\gamma '(t)= \gamma (\e t)$ for all $t\in (-1,1)$.
For such a real analytic arc $\gamma'$, $\phi_m\circ {\gamma}'$ is a real analytic arc on $V_m\cap\phi_m(U)$ since $\phi_m$ is an arc-analytic map (cf. Theorem \ref{Main}). Thus $\phi_m\circ {\gamma}$ is a  real analytic arc on $(V_m,0)$. On the other hand $\phi_m^{-1}$ is arc-analytic thus the same procedure applies. This shows that $\phi_m$ induces a bijection between the space of real analytic arcs on $(X,0)$ and the space of real analytic arcs on $(V_m,0)$.

Now if ${\gamma}$ is a real analytic arc on $(X,0)$ then 
$$
\|\phi_m( {\gamma}(t))- {\gamma}(t)\|\leq  C^m\|{\gamma}(t)\|^m\leq {C(m,\gamma)}|t|^m$$
 for all $t$ small enough and some positive constant $C(m,\gamma)$ depending only on $m$ and $\gamma$. This shows that $x_i\circ\phi_m\circ {\gamma}(t)$  has the same Taylor expansion as $x_i\circ{\gamma}$  up to order $m-1$. Thus $\phi_m$ induces the identity map between the space of $m-1$-truncations of real analytic arcs on $(X,0)$ and the space  of $m-1$-truncations of real analytic arcs on $(V_m,0)$. (We need to shift the index $m\to m+1$ to get the map of  the statement of the theorem.)
 
 If $(X,0)$ is complex analytic  then any real analytic arc germ ${\gamma} : (\R,0)\lgw (X,0)$ extends uniquely to 
 a complex analytic arc germ ${\gamma}_\C : (\C ,0)\lgw (X,0)$ (both are given by 
 exactly the same power series), and similarly for the arc-germs in $(V_m,0)$.   
 Of course, any complex arc germ in $(X,0)$, resp. in  $(V_m,0)$,   is such extension of a real arc germ. 
 Thus $\mathcal{A}_m^{\C}(X)=\mathcal{A}_m^{\C}(V_m)$ follows from $\mathcal{A}_m^{\R}(X)=\mathcal{A}_m^{\R}(V_m)$.
 
 \begin{rem}\label{exnonconv} 
By a result of M. Greenberg (cf. \cite{Gr} or \cite{Sc} for the analytic case) for a given analytic germ $(X,0)\subset (\K^n,0)$ there exists a constant $a=a_X>0$ such that for every integer $m$ we have that
$$\mathcal A_m^\K(X)=\pi_{m}\left(\mathcal B_{am}^\K(X)\right)$$ 
where $\mathcal B_k^\K(X)$ denotes the space of $k$-jets on $(X,0)$, i.e. the space of $\K$-analytic arcs on $(\K^n,0)$ whose contact order with $(X,0)$ is at least $k+1$, and $\pi_{m}$ is the truncation map at order $m$, i.e. the map sending an arc on $(\K^n,0)$ onto its truncation at order $m$. If we denote by $(X',0)$ an analytic germ defined by equations that coincide with the equations defining $(X,0)$ up to order $am$, then we obviously have that
$$\mathcal B_{am}^\K(X)=\mathcal B_{am}^\K(X').$$
But while $\mathcal A_m^{\K}(X)$ is  the truncation of $\mathcal B_{am}^\K(X)$, $\mathcal A_m^{\K}(X')$ has no reason in general to be equal to the truncation of $\mathcal B_{am}^\K(X')$ since the constant $a_{X'}$ of Greenberg's Theorem may be strictly greater than $a=a_X$. Thus we cannot prove, using Greenberg's Theorem, that 
$\mathcal A_m^\K(X)$ is equal to $\mathcal A_m^\K(X')$ when $(X',0)$ is an analytic germ whose equations coincide with those of $(X,0)$ up to a high order.\\ 
\\
In fact  a high order of tangency of two analytic germs
does not guarantee that the spaces of truncated arcs associated with these germs are equal. 
We can explicitly show this on the following example. Let 
$$f(x,y,z)=z^2-xy^4,\  f_k(x,y,z)=z^2-x(y^4+x^{2k})$$
 and define $X=\{(x,y,z): f(x,y,z)=0\}, 
X_k=\{(x,y,z):f_k(x,y,z)=0\}$ for any positive integer 
$k$ which is not divisible by $2.$ Then the arc $\gamma$ given by $\gamma(t) = (t,0,0)$ satisfies $\gamma\in \mathcal{A}_1^{\K}(X)$
but in both cases real or complex $\gamma\notin\mathcal{A}_1^{\K}(X_k).$ Indeed, let
$\tilde{\gamma}$ be any arc whose truncation up to order $1$ equals $\gamma.$ Then after substituting $\tilde{\gamma}$ to $x(y^4+x^{2k})$ and to $z^2,$ we obtain power series with odd and even order of zero, respectively, so $f_k\circ\tilde{\gamma}\neq 0$. Thus $\mathcal{A}_1^{\K}(X_k)\neq \mathcal{A}_1^{\K}(X),$
although $2k$-truncations of $f$  and $f_k$ are equal and the multiplicities of $X$ and $X_k$ at $0$ are also equal for $k$ large enough.\end{rem}
%%%%%%%%%%%%%%%%%%%%%%%%%%%%%%%%%%%%%%%%%%%%%%%%%%%%%%%

\end{document}